\documentclass[a4,12pt]{amsart}
\usepackage{amssymb,amstext,amsmath,amscd,amsthm,amsfonts,enumerate,latexsym}
\usepackage{color}
\usepackage[dvipdfmx]{graphicx}
\usepackage[all]{xy}

\usepackage[dvipdfmx]{hyperref}

\theoremstyle{plain}
\newtheorem{thm}{Theorem}[section]

\newtheorem*{thm*}{Theorem}
\newtheorem*{cor*}{Corollary}

\newtheorem{prop}[thm]{Proposition}

\newtheorem{lem}[thm]{Lemma}
\newtheorem{cor}[thm]{Corollary}

\newtheorem{claim}{Claim}
\newtheorem*{claim*}{Claim}

\theoremstyle{definition}
\newtheorem{defn}[thm]{Definition}

\newtheorem{ex}[thm]{Example}

\theoremstyle{remark}

\numberwithin{equation}{thm}

\def\Im{\operatorname{Im}}

\def\Ker{\operatorname{Ker}}

\def\bbZ{\mathbb{Z}}

\def\Max{\operatorname{Max}}

\def\m{\mathfrak m}

\def\p{\mathfrak p}

\newcommand{\rme}{\mathrm{e}}

\newcommand{\rmH}{\mathrm{H}}

\newcommand{\rmK}{\mathrm{K}}

\newcommand{\rmQ}{\mathrm{Q}}

\newcommand{\fkm}{\mathfrak{m}}

\newcommand{\fkp}{\mathfrak{p}}

\newcommand{\mapright}[1]{%
\smash{\mathop{%
\hbox to 1cm{\rightarrowfill}}\limits^{#1}}}

\newcommand{\mapleft}[1]{%
\smash{\mathop{%
\hbox to 1cm{\leftarrowfill}}\limits_{#1}}}

\def\depth{\operatorname{depth}}

\def\height{\mathrm{ht}}
\def\Spec{\operatorname{Spec}}

\title[Topics on strict closure of rings]{Topics on strict closure of rings}

\author[Naoki Endo]{Naoki Endo}
\address{Department of Mathematics, Faculty of Science Division II, Tokyo University of Science, 1-3 Kagurazaka, Shinjuku, Tokyo 162-8601, Japan}
\email{nendo@rs.tus.ac.jp}
\urladdr{https://www.rs.tus.ac.jp/nendo/}

\author[Shiro Goto]{Shiro Goto}
\address{Department of Mathematics, School of Science and Technology, Meiji University, 1-1-1 Higashi-mita, Tama-ku, Kawasaki 214-8571, Japan}
\email{shirogoto@gmail.com}

\author[Ryotaro Isobe]{Ryotaro Isobe}
\address{Department of Mathematics and Informatics, Graduate School of Science, Chiba University, 1-33 Yayoi-cho, Inage-ku, Chiba 263-8522, Japan}
\email{r.isobe.math@gmail.com}

\thanks{2020 {\em Mathematics Subject Classification.} 13A15, 13B22, 13B30.}
\thanks{{\em Key words and phrases.} Strictly closed ring, Arf ring, weakly Arf ring}

\thanks{The first author was partially supported by JSPS Grant-in-Aid for Young Scientists 20K14299 and JSPS Overseas Research Fellowships. The second author was partially supported by JSPS Grant-in-Aid for Scientific Research (C) 16K05112. The third author was partially supported by Grant-in-Aid for JSPS Fellows 20J10517.}

\begin{document}

\maketitle

\setlength{\baselineskip} {16pt}

\begin{abstract}
In 1971, J. Lipman \cite{L} introduced the notion of strict closure of a ring in another, and established the underlying theory in connection with a conjecture of O. Zariski.
In this paper, for further developments of the theory, we investigate three different topics related to strict closure of rings. 
The first one concerns construction of the closure, and the second one is the study regarding the question of whether the strict closedness is inherited under flat homomorphisms. We finally handle the question of when the Arf closure coincides with the strict closure. Examples are explored to illustrate our theorems.


\end{abstract}



\section{Introduction}\label{intro}

In this paper, for an extension of commutative rings $S/R$, we investigate the structure of an intermediate ring $R^*$, called {\it the strict closure of $R$ in $S$}, between $R$ and $S$ consisting of those elements $\alpha$ in $S$ such that $\alpha\otimes 1=1\otimes \alpha$ in $S\otimes_R S$. We say that $R$ is {\it strictly closed} in $S$, if $R=R^*$. Since $[R^*]^*=R^*$ in $S$, $R^*$ is strictly closed in $S$. In addition, if $T$ is an intermediate ring between $R$ and $S$, then $R^* \subseteq T^*$ in $S$ (\cite[Section 4, p.672]{L}). Hence $[-]^*$ is a closure operation. The ring $R$ is called {\it strictly closed}, provided $R$ is strictly closed in $\overline{R}$, where $\overline{R}$ denotes the integral closure of $R$ in its total ring $\operatorname{Q}(R)$ of fractions. 
The notion of strictly closed rings was introduced by J. Lipman \cite{L} in 1971, and he established the underlying theory, revealed a connection with Zariski's theory of saturation of rings.

A typical example of strictly closed rings is the Stanley-Reisner algebra of a simplicial complex, and we also encounter strictly closed rings arising from the Rees algebras of certain ideals (\cite[Theorem 3.3, Corollary 2.4]{EG}). Moreover, except the modular case, the invariant subrings of strictly closed rings under a finite group action are strictly closed (\cite[Corollary 13.6]{CCCEGIM}). Thus, strictly closed rings appear in diverse branches of fields related to commutative algebra. Among them, this is what matters most is a relation between the theory of Arf rings, defined by J. Lipman in the same paper \cite{L}.

Let $R$ be a Noetherian semi-local ring and assume that $R$ is {\it a Cohen-Macaulay ring of dimension one}, i.e., for every $M \in \Max R$, the local ring $R_M$ is Cohen-Macaulay and of dimension one. We say that $R$ is {\it an Arf ring}, if the following conditions are satisfied (\cite[Definition 2.1]{L}):
\begin{enumerate}[$(1)$]
\item Every integrally closed ideal $I$ in $R$ that contains a non-zerodivisor has {\it a principal reduction}, i.e. $I^{n+1} = a I^n$ for some $n \ge 0$ and $a \in I$. 
\item Let $x,y,z \in R$ such that $x$ is a non-zerodivisor on $R$ and $\frac{y}{x}, \frac{z}{x} \in \overline{R}$. Then $\frac{yz}{x} \in R$.
\end{enumerate}


Suppose now that $\overline{R}$ is a module-finite extension of $R$. Then, among all the Arf rings between $R$ and its integral closure $\overline{R}$, there is a smallest one, which is called {\it the Arf closure of $R$}. The connection between Arf and strict closures is expressed as a conjecture of O. Zariski. As is mentioned in \cite{L}, O. Zariski conjectured that the strict closure of $R$ in $\overline{R}$ coincides with the Arf closure of $R$. As for this conjecture, O. Zariski proved the {\it if} part, and J. Lipman gave an affirmative answer for the converse, when $R$ contains a field (\cite[Proposition 4,5, Theorem 4.6]{L}). Until this was settled in \cite[Theorem 4.4]{CCCEGIM} with full generality, it seems that Zariski's conjecture has been open for more than one half of a century.





Recently, the authors of \cite{CCCEGIM} introduced the notion of weakly Arf rings, towards a theory of higher dimensional Arf rings, by weakening the defining conditions of Arf rings. For an arbitrary commutative ring $R$, the ring $R$ is called {\it weakly Arf}, if $R$ satisfies condition $(2)$ in the above definition. As J. Lipman predicted in \cite{L} and as the authors of \cite{CCCEGIM} are  now developing the theory of weakly Arf rings, there might have a strong connection between the strict closedness and the weak Arf property. In fact, for an arbitrary commutative ring $R$, there exists {\it the weakly Arf closure $R^a$ of $R$} which is a smallest weakly Arf ring between $R$ and $\overline{R}$, and the following holds.


\begin{thm}[{\cite[Corollary 7.7]{CCCEGIM}}]\label{1.2}
Let $R$ be a Noetherian ring. Suppose that $\overline{R}$ is a finitely generated $R$-module and one of the following conditions:
\begin{enumerate}[$(1)$]
\item
$R$ contains an infinite field.
\item
${\rm ht}_R M\ge 2$ for every $M\in \Max R$.
\end{enumerate}
If the weakly Arf closure $R^a$ of $R$ satisfies the condition $(S_2)$ of Serre, then $R^a=R^*$.
\end{thm}

In contrast, the first and second authors \cite{EG} gave a practical method of construction of strictly closed rings. In general, it is difficult to compute the strict closure of rings. They constructed some concrete examples of them, and showed that the Stanley-Reisner rings of simplicial complexes (resp. $F$-pure rings satisfying $(S_2)$) are strictly closed.
The reader may consult with \cite{L, CCCEGIM, EG} about further study of strictly closed rings.

In the present paper, we investigate three topics related to strict closure of rings. The first one concerns construction of strictly closed rings. In \cite{CCCEGIM, EG}, the authors handled only the case where the overring $S$ is the integral closure $\overline{R}$. In Section 2, we consider an arbitrary extension of commutative rings $S/R$, and give an upper bound for the strict closure of $R$ in $S$, using a presentation matrix of $S$; see Theorem \ref{2.6}. 

As a second, in Section 3, we explore the question of how the strict closedness is inherited under flat homomorphisms. J. Lipman proved in \cite{L} that, for a faithfully flat homomorphism $\varphi: R \to S$ of rings, if $S$ is strictly closed, then so is $R$. Theorem \ref{3.1} shows that $S$ is strictly closed if and only if so is $R$, once the  homomorphism $\varphi: R\to S$ satisfies additional assumptions. 

Notice that the weakly Arf closure $R^a$ coincides with Arf closure when $R$ is a Cohen-Macaulay semi-local ring of dimension one and $\overline{R}$ is a finitely generated $R$-module. Moreover, since $R^*$ is a weakly Arf ring, $R^a\subseteq R^*$ holds (\cite[Section 7]{CCCEGIM}).
Finally, Section 4 deals with a generalization of Zariski's conjecture, that is the question of when the equality $R^a=R^*$ holds. Theorem \ref{1.2} requires that $R^a$ satisfies $(S_2)$. In the present paper, we focus on the case where $R^a$ does not satisfy the condition, while introducing examples.

Throughout this paper, unless otherwise specified, let $R$ be an arbitrary commutative ring, and $\overline{R}$ the integral closure of $R$ in its total ring $\operatorname{Q}(R)$ of fractions.







\section{Upper bounds for strict closures}

Let $S/R$ be an extension of commutative rings, and we set
$$
R^{*}=\left\{ \alpha \in S \mid \alpha \otimes 1 = 1 \otimes \alpha \text{ in } S \otimes_RS\right\}
$$
which is a kernel of an $R$-linear map 
$$
\sigma: S \to S \otimes_RS, \ \  \alpha \mapsto \alpha \otimes 1 - 1 \otimes \alpha.
$$
Then $R^*$ forms a subring of $S$ containing $R$, which is called {\it the strict closure} of $R$ in $S$. We say that $R$ is {\it strictly closed} in $S$, if $R=R^{*}$. Since $[R^*]^*=R^*$ in $S$, $R^*$ is strictly closed in $S$. Moreover, if $T$ is an intermediate ring between $R$ and $S$, then $R^* \subseteq T^*$ in $S$; see \cite[Section 4, p.672]{L}. Thus $[-]^*$ is actually a closure operation. We simply say that $R$ is {\it strictly closed}, when $R$ is strictly closed in $\overline{R}$, where $\overline{R}$ denotes the integral closure of $R$ in its total ring $\operatorname{Q}(R)$ of fractions. The reader may consult with \cite{L, CCCEGIM, EG} about properties of strict closures.

We begin with the following.

\begin{lem}\label{2.1}
Let $\varphi : R \to S$ be a homomorphism of rings. Suppose that $S$ is free as an $R$-module and that $1\in R$ is a part of a free basis of $S$. Then the sequence
$$
0 \to R \overset{\varphi}{\to} S \overset{\sigma}{\to} S\otimes_R S
$$
is exact as an $R$-module. Hence, $\varphi(R)$ is strictly closed in $S$.
\end{lem}

\begin{proof}
Notice that $\varphi$ is injective. Let $\{e_i\}_{i \in \Lambda}$ be an $R$-free basis of $S$ such that $e_{i_0} = 1$ for some $i_0 \in \Lambda$. We then have $S = \sum_{i \in \Lambda}Re_i = \bigoplus_{i \in \Lambda}Re_i$, so that
$$
S \otimes_R S = \sum_{i, j \in \Lambda}R(e_i \otimes e_j) = \bigoplus_{i, j \in \Lambda}Re_i \otimes_R Re_j.
$$
Let $\alpha \in \Ker \sigma$ and write $\alpha = (\alpha_i)_{i \in \Lambda} = \sum_{i \in \Lambda}\alpha_ie_i$, where $\alpha_i \in R$ for all $i \in \Lambda$. Then 
$$
\alpha \otimes 1 = \sum_{i \in \Lambda}\alpha_i(e_i \otimes 1)  \ \ \text{and} \ \ 1 \otimes \alpha =  \sum_{i \in \Lambda}\alpha_i(1 \otimes e_i)
$$
in $S \otimes_R S$. Thus, $\alpha \otimes 1 = 1 \otimes \alpha$ if and only if $\alpha_i=0$ for all $i \in \Lambda$ such that $i\neq i_0$. Therefore $\Ker \sigma \subseteq \Im \varphi$, and hence $\Ker \sigma = \Im \varphi$. This completes the proof.
\end{proof}

The next gives an upper bound for strict closures of rings.

\begin{prop}\label{2.2}
Let $S/R$ be an extension of commutative rings. Then 
$$
R^* \subseteq R + MS \ \ \text{in} \ \ S
$$
for every $M \in \Max R$.
\end{prop}
 
\begin{proof}
Let $M \in \Max R$. We may assume that $MS \ne S$. By applying the functor $R/M \otimes_R (-)$ to the exact sequence $0  \to R^* \overset{i}{\to} S \overset{\sigma}{\to} S\otimes_R S$ of $R$-modules, we get a complex
$$
\xymatrix{
 R/M \otimes_R R^* \ar[r]^{1 \otimes i}  & R/M \otimes_R S \ar[rd]^{\sigma} \ar[r]^{1 \otimes \sigma} & R/M \otimes_R(S \otimes_R S) \ar[d]^{\cong}  \\
  & & (R/M \otimes_R S) \otimes_{R/M} (R/M \otimes_R S) 
}
$$
of $R/M$-vector spaces, where the vertical map is a canonical isomorphism.
Let $\alpha \in R^*$. We then have 
$1 \otimes \alpha \in \Ker (R/M \otimes_R S \overset{\sigma}{\to} (R/M \otimes_R S) \otimes_{R/M} (R/M \otimes_R S))$. By applying Lemma \ref{2.1} to the homomorphism $R/M \otimes_R R \overset{1 \otimes i}{\longrightarrow} R/M \otimes_R S$ of rings, we get
$$
1 \otimes \alpha = 1 \otimes a \ \ \ \text{in} \ \ R/M \otimes_R S
$$
for some $a \in R$. Therefore, $\alpha -a \in MS$, so that $R^* \subseteq R + MS$, as desired. 
\end{proof}

Hence we have the following.

\begin{cor}\label{2.3}
Let $S/R$ be an extension of commutative rings. Suppose that $R$ is a local ring $($not necessarily Noetherian$)$ with maximal ideal $\fkm$. Then
$$
R^* \subseteq R + \m S \ \ \text{in} \ \ S.
$$
In particular, $R$ is strictly closed in $S$, provided $\m S \subseteq R$. 
\end{cor}

\begin{cor}
Let $(R, \m)$ be a Cohen-Macaulay local ring with $\dim R=1$, possessing a canonical module $\rmK_R$. Suppose that there exists an $R$-submodule $K$ of $\rmQ(R)$ such that $R \subseteq K \subseteq \overline{R}$ and $K \cong \rmK_R$ as an $R$-module. If $R$ is an almost Gorenstein ring in the sense of \cite{GMP}, then $R$ is strictly closed in $S$, where $S = R[K]$. 
\end{cor}

\begin{proof}
This follows from Corollary \ref{2.3} and \cite[Theorem 3.11]{GMP}.
\end{proof}

Strict closedness depends on its extension of rings as we show below.  

\begin{ex}
Let $V = k[[t]]$ be the formal power series ring over a field $k$. We set $R = k[[t^3, t^8, t^{13}]]$ and $S=k[[t^3, t^5]]$. Then $S = R + Rt^5 + Rt^{10}$, so that $R$ is strictly closed in $S$, because $(t^3, t^8, t^{13}) S = t^3 S \subseteq R$. However, since $R$ is not an Arf ring, $R$ is not strictly closed in $V$; see \cite[Proposition 4.5]{L}.
\end{ex}

For an extension of commutative rings $S/R$, if $S$ is finitely presented as an $R$-module, then the strict closure $R^*$ in $S$ has a better upper bound than Corollary 2.3.

\begin{thm}\label{2.6}
Let $S/R$ be an extension of commutative rings. Suppose that $S$ is a module-finite extension of $R$ that is generated by $\{1, f_1, \ldots, f_n\}$, where $n>0$ and $f_i \in S$ for each $1 \le i \le n$. Moreover, we assume that $S$ has a presentation 
$$
R^{\oplus q} \overset{\Bbb M}{\longrightarrow} R^{\oplus (n+1)} \overset{\varepsilon}{\longrightarrow} S \longrightarrow 0
$$
of $R$-modules, where $q > 0$ and $\varepsilon = \left[1 \ f_1 \ \cdots \ f_n\right]$. Then 
$$
R^* \subseteq R + J  \ \ \text{in} \ \ S
$$
where $J$ denotes an ideal of $S$ generated by all the entries of the first row of $\Bbb M$. 
\end{thm}

\begin{proof}
By applying the functor $S \otimes_R(-)$ to the presentation of $S$, we have the following diagram 
$$
\xymatrix{
S \otimes_R R^{\oplus q} \ar[r]^{S \otimes {\Bbb M} \ \ \ }  & S \otimes_R R^{\oplus (n+1)}  \ar[r]^{\ \ \ S \otimes \varepsilon} & S \otimes_R S \ar[r] & 0 \\
S^{\oplus q} \ar[u]_{\cong} \ar[r]^{\Bbb M \ \ }& S^{\oplus (n+1)}\ar[u]_{\cong} & 
}
$$
of $S$-modules, where the vertical maps are canonical. 
Let $\alpha \in R^*$, where $R^*$ denotes the strict closure of $R$ in $S$, and write 
$$
\alpha = \alpha_0 + \alpha_1f_1 + \cdots + \alpha_n f_n
$$
for some $\alpha_i \in R$. Set $\beta = \alpha_1f_1 + \cdots + \alpha_n f_n$. Then, since $R \subseteq R^*$, we have $\beta \in R^*$. Therefore
$$
\beta \otimes 1 - 1 \otimes \beta = \beta \otimes 1 - \sum_{i = 1}^n \alpha_i (1 \otimes f_i) = 0
$$
in $S \otimes_R S$. By setting $\{e_i\}_{0 \le i \le n}$ the standard basis of $R^{\oplus (n+1)}$, we obtain 
$$
\beta \otimes e_0 - \sum_{i = 1}^n \alpha_i (1 \otimes e_i) \in \Ker (S \otimes \varepsilon)
$$
which yields that 
$$
\left(\begin{smallmatrix}
\beta \\
-\alpha_1 \\
\vdots \\
- \alpha_n
\end{smallmatrix}\right) \in \Im (S^{\oplus q} \overset{\Bbb M}{\longrightarrow} S^{\oplus (n+1)}).
$$
This implies $\beta \in J$, the ideal of $S$ generated by all the entries of the first row of $\Bbb M$. Hence $\alpha = \alpha_0 + \beta \in R + J$, as desired. 
\end{proof}

\begin{cor}\label{2.7}
We maintain the notation as in Theorem \ref{2.6}. If $f_i f_j \in R$ for each $1 \le i, j \le n$, then $R$ is strictly closed in $S$.
\end{cor}

\begin{proof}
Since $S = R + \sum_{i=1}^n Rf_i$, we have $R:S = \bigcap_{i=1}^n[R:f_i]$. Let $\alpha \in \Ker \varepsilon$ and write $\alpha = \sum_{i=0}^n\alpha_ie_i$, where $\{e_i\}_{0 \le i \le n}$ denotes the standard basis of $R^{\oplus (n+1)}$ and $\alpha_i \in R$ for all $0 \le i \le n$. We then have $\alpha_0 = - \sum_{i=1}^n\alpha_if_i$, whence
$$
\alpha_0 f_j = - \sum_{i=1}^n(\alpha_if_i)f_j  = - \sum_{i=1}^n\alpha_i(f_if_j) \in R 
$$
for all $1 \le j \le n$. Thus, $\alpha_0 \in R:S$. Hence $J \subseteq R:S \subseteq R$, so that $R^* \subseteq R + J \subseteq R$. 
\end{proof}

We now explore some examples in order to illustrate Theorem \ref{2.6}.

\begin{ex}
Let $S=k[X, Y]$ be the polynomial ring over a field $k$. Let $n \ge 6$ be an integer and set
$$
R = k[X^{n-i}Y^i \mid 0 \le i \le n, \ i \ne 1, 3]
$$
in $S$. Then $R$ is a strictly closed Cohen-Macaulay ring with $\dim R=2$.
\end{ex}

\begin{proof}
Notice that $\overline{R}= k[X^{n-i}Y^i \mid 0 \le i \le n] = R + R X^{n-1}Y + RX^{n-3}Y^3$. By Corollary \ref{2.7}, $R$ is strictly closed in $\overline{R}$, because $(X^{n-1}Y)^2, (X^{n-1}Y)(X^{n-3}Y^3), (X^{n-3}Y^3)^2 \in R$. 
Moreover, since $X^n,Y^n$ forms a regular sequence on $R$, the ring $R$ is Cohen-Macaulay and of dimension $2$. 
\end{proof}


\begin{ex}
Let $S=k[X, Y]$ be the polynomial ring over a field $k$. Let  $\ell \ge 3$ be an integer and set 
$$
R = k[X, Y]/(Y^{\ell}).
$$
We denote by $x, y$ the images of $X, Y$ in $R$, respectively. Note that $R$ is a finitely generated free $D$-module with a $D$-free basis $\{1, y, y^2, \ldots, y^{\ell -1}\}$, where $D = k[x]$.  

For each $n > 0$, we define
$$
z_n = \frac{y}{x^n} \in \rmQ(R) \ \ \ \text{and} \ \ \ S_n = D[z_n].
$$ 
Then, $S_n$ is an intermediate ring between $R$ and $\overline{R}$, and 
$$
R^* = R + D\frac{y^2}{x^n} + D\frac{y^3}{x^{2n}} + \cdots + D\frac{y^{\ell -1}}{x^{(\ell -2)n}}  \ \ \text{in} \ \ S_n.
$$ 
\end{ex}

\begin{proof}
Notice that $x$ is transcendental over $k$. 
Let $K = \rmQ(D)$ be the total ring of fractions of $D$. Then $K \subseteq \rmQ(R)$, so that $\rmQ(R)=K[y]$ and $\{y^i\}_{0 \le i \le \ell-1}$ forms a $K$-free basis of $\rmQ(R)$. We consider the $K$-algebra map
$\varphi: K[y] \to K$ such that $\varphi(y) = 0$. Let $f \in \overline{R}$ and write $f = \sum_{i=0}^{\ell-1} c_i y^i$ with $c_i \in K$. Since $f \in \overline{R}$, we then have $c_0 = \varphi(f) \in \overline{D} = D$. Hence
$$
\overline{R} = D + K y + K y^2 + \cdots + K y^{\ell -1}.
$$
For each $n > 0$, we set
$$
z_n = \frac{y}{x^n} \in \rmQ(R) \ \ \ \text{and} \ \ \ S_n = D[z_n] = D + Dz_n + \cdots + Dz_n^{\ell-1}.
$$ 
The relation $x^n z_n = y$ and $z_n \in Ky$ induce $Dy \subseteq Dz_n \subseteq Ky$, so that $R \subseteq S_n \subseteq \overline{R}$. We consider the $R$-linear map $\varepsilon: R^{\oplus \ell} \to S_n$ defined by 
$$
\varepsilon\left(\begin{smallmatrix}
a_0 \\
a_1 \\
\vdots \\
a_{\ell -1}
\end{smallmatrix}\right) =  \sum_{i=0}^{\ell -1} a_iz^i = a_0 + a_1 z + \cdots + a_{\ell -1 }z^{\ell -1}.
$$
Then, it is straightforward to check that 
$$
\Ker \varepsilon = \left<
\left(\begin{smallmatrix}
y \\
-x^n \\
0 \\
0 \\
\vdots \\
0
\end{smallmatrix}\right), 
\left(\begin{smallmatrix}
0 \\
y \\
-x^n \\
0 \\
\vdots \\
0
\end{smallmatrix}\right), 
\cdots, 
\left(\begin{smallmatrix}
0 \\
0 \\
\vdots \\
0 \\
y \\
-x^n
\end{smallmatrix}\right), 
\left(\begin{smallmatrix}
0 \\
0 \\
\vdots \\
0 \\
0 \\
y
\end{smallmatrix}\right)
 \right>
$$
as an $R$-module. Therefore, by Theorem \ref{2.6}, we have
$$
R^* \subseteq R + yS_n = R + D\frac{y^2}{x^n} + D\frac{y^3}{x^{2n}} + \cdots + D\frac{y^{\ell -1}}{x^{(\ell -2)n}}
$$
in $S_n$. On the other hand, we consider the following diagram
$$
\xymatrix{
S_n \otimes_R R^{\oplus \ell} \ar[r]^{S_n \otimes {\Bbb M}  }  & S_n \otimes_R R^{\oplus \ell}  \ar[r]^{S_n \otimes \varepsilon} & S_n \otimes_R S_n \ar[r] & 0 \\
S_n^{\oplus \ell} \ar[u]_{\cong} \ar[r]^{\Bbb M \ \ }& S_n^{\oplus \ell}\ar[u]_{\cong} & 
}
$$
of $S_n$-modules, where the vertical maps are canonical and the matrix $\Bbb M$ is determined by the generators of $\Ker \varepsilon$. Because of the following relations 
\begin{eqnarray*}
\left(\begin{smallmatrix}
\frac{y^2}{x^n} \\
0 \\
-x^n \\
0 \\
\vdots \\
0
\end{smallmatrix}\right) 
&=&
\frac{y}{x^n} 
\left(\begin{smallmatrix}
y \\
-x^n \\
0 \\
0 \\
\vdots \\
0
\end{smallmatrix}\right)
+ \left(\begin{smallmatrix}
0 \\
y \\
-x^n \\
0 \\
\vdots \\
0
\end{smallmatrix}\right) \\
\left(\begin{smallmatrix}
\frac{y^3}{x^{2n}} \\
0 \\
0 \\
-x^n \\
\vdots \\
0
\end{smallmatrix}\right) 
&=&
\left(\frac{y}{x^n} \right)^2
\left(\begin{smallmatrix}
y \\
-x^n \\
0 \\
0 \\
\vdots \\
0
\end{smallmatrix}\right)
+ 
\frac{y}{x^n} 
\left(\begin{smallmatrix}
0 \\
y \\
-x^n \\
0 \\
\vdots \\
0
\end{smallmatrix}\right)
+
\left(\begin{smallmatrix}
0 \\
0 \\
y \\
-x^n \\
\vdots \\
0
\end{smallmatrix}\right) \\
&\vdots& \\
\left(\begin{smallmatrix}
\frac{y^{\ell-1}}{x^{(\ell-2)n}} \\
0 \\
0 \\
0 \\
\vdots \\
-x^n
\end{smallmatrix}\right) 
&=&
\left(\frac{y}{x^n} \right)^{\ell -2}
\left(\begin{smallmatrix}
y \\
-x^n \\
0 \\
0 \\
\vdots \\
0
\end{smallmatrix}\right)
+ 
\left(\frac{y}{x^n} \right)^{\ell -3}
\left(\begin{smallmatrix}
0 \\
y \\
-x^n \\
0 \\
\vdots \\
0
\end{smallmatrix}\right)
+ \cdots + 
\frac{y}{x^n}
\left(\begin{smallmatrix}
0 \\
0 \\
\vdots \\
y \\
-x^n \\
0
\end{smallmatrix}\right)
+
\left(\begin{smallmatrix}
0 \\
0 \\
\vdots \\
0 \\
y \\
-x^n 
\end{smallmatrix}\right),
\end{eqnarray*}
we then have 
$$
\left(\frac{y^i}{x^{(i-1)n}}\right) \otimes 1 = 1 \otimes \left(\frac{y^i}{x^{(i-1)n}}\right) \ \ \ \text{for every} \ \ 2 \le i \le \ell-1
$$
in $S_n \otimes_R S_n$. Hence we get 
$$
R^*  = R + D\frac{y^2}{x^n} + D\frac{y^3}{x^{2n}} + \cdots + D\frac{y^{\ell -1}}{x^{(\ell -2)n}} \ \ \text{in} \ \ S_n
$$
as claimed. 
\end{proof}


\section{Flat base change of strictly closed rings}

In this section, we explore the question of how the strict closedness is inherited under flat homomorphisms. 
Let $A/R$ be an extension of commutative rings, and let $\varphi: R \to S$ be a flat homomorphism. With this notation, J. Lipman proved in \cite[Corollary 4.4]{L} that the strict closure $[S\otimes_RR]^*$ of $S\otimes_RR$ in $S \otimes_R A$ coincides with $S \otimes_R R^*$, where $R^*$ denotes the strict closure of $R$ in $A$. 
Therefore, for a faithfully flat homomorphism $\varphi: R \to S$, if $S$ is strictly closed, then so is $R$.


The goal of this section is to prove the following. 

\begin{thm}\label{3.1}
Let $\varphi: R \to S$ be a faithfully flat homomorphism of Noetherian rings. For each homomorphism $R \to R'$ of rings such that $R'$ is a finitely generated $R$-module, we assume  $[S \otimes_R R']_P/\p[S \otimes_R R']_P$ is regular for every $P \in \Spec (S \otimes_R R')$, where $\p = P \cap R'$. Suppose that $R$ is reduced and $\overline{R}$ is a finitely generated $R$-module. Then the following assertions hold true, where $\beta: S \otimes_R\rmQ(R) \to \rmQ(S)$ denotes the canonical injection. 
\begin{enumerate}[$(1)$]
\item $S$ is reduced. 
\item $\beta(S \otimes_R \overline{R}) = \overline{S}$ in $\rmQ(S)$. 
\item $\beta(S \otimes_RR^*) = S^*$ in $\rmQ(S)$. 
\item $R$ is strictly closed if and only if $S$ is strictly closed. 
\end{enumerate}
\end{thm}

To prove Theorem \ref{3.1}, we need an auxiliary. For each integer $n$, a Noetherian ring $R$ satisfies {\it the condition $(S_n)$} (resp. $(R_n)$) {\it of Serre}, if $\depth R_P \ge \min \{n, \dim R_P\}$ for all $P \in \Spec R$ (resp. $R_P$ is a regular local ring for all $P \in \Spec R$ with $\dim R_P \le n$).

\begin{lem}\label{3.2}
Let $\varphi: R \to S$ be a flat homomorphism of Noetherian rings. Suppose that $S_P/\p S_P$ is regular for every $P \in \Spec S$, where $\p = P \cap R$. For each $n \ge 1$, if $R$ satisfies the conditions $(S_n)$ and $(R_{n-1})$ of Serre, then so is $S$.
\end{lem}

\begin{proof}
Suppose that $R$ satisfies the conditions $(S_n)$ and $(R_{n-1})$ of Serre. For $P \in \Spec S$, we set $\p = P \cap R$. The map $\varphi: R \to S$ induces the flat local homomorphism $R_\p \to S_P$ of rings. Firstly we will show that $\depth S_P \ge \min \{n, \dim S_P\}$. Indeed, we may assume $\depth S_P \le n-1$. Since $\depth S_P = \depth R_\p + \depth S_P/\p S_P$, we have $\depth R_\p \le n-1$, so that $\dim R_\p \le n-1$. 
Therefore the local ring $S_P$ is regular, because $R$ satisfies $(R_{n-1})$ and $S_P/\p S_P$ is regular. Hence $\depth S_P \ge \min \{n, \dim S_P\}$. 
Secondly, let $P \in \Spec S$ such that $\dim S_P \le n-1$. We set $\p = P \cap R$.  Then $\dim R_\p \le n$, because the induced homomorphism $R_\p \to S_P$ is flat and local. Hence $S_P$ is regular. 
\end{proof}

We are now ready to prove Theorem \ref{3.1}

\begin{proof}[Proof of Theorem \ref{3.1}]
By Lemma \ref{3.2}, the ring $S$ is reduced. Let $C = S \otimes_R\overline{R}$. Look at the commutative diagram
$$
\xymatrix{
 S \otimes_R \rmQ(R) \ar[r]^{\ \ \beta}  & \rmQ(S)  \\
\hspace{-2em}C= S \otimes_R \overline{R} \ar[r]^{\ \ \beta}\ar[u]  & \overline{S}  \ar[u]& \\
 S \otimes_R R \ar[r]^{\ \ \cong}\ar[u] & S \ar[u] & 
}
$$
of $R$-algebras, where the vertical maps are canonical. For $P \in \Spec C$, we set $\p = P \cap \overline{R}$. 
Then our hypothesis guarantees that the ring $C_P/\p C_P$ is regular. By applying Lemma \ref{3.2} to the homomorphism $\overline{R} \to C=S \otimes_R \overline{R}$, we conclude that $C$ is normal. Hence the equality $\beta(S \otimes_R \overline{R}) = \overline{S}$ holds in $\rmQ(S)$. 
By taking the functor $S \otimes_R (-)$ to the exact sequence $0 \to R^* \to \overline{R} \overset{\sigma} \to \overline{R} \otimes_R\overline{R}$ of $R$-modules, we get the commutative diagram
$$
\xymatrix{
 0 \ar[r] & S \otimes_R R^* \ar[r] & C \ar[r]^{\sigma \ \ \ }  \ar[d]^{\cong} & C \otimes_SC \ar[d]^{\cong} \\
 0 \ar[r] & S^* \ar[r]  & D \ar[r]^{\sigma \ \ \ } & D \otimes_SD
 }
$$
of $S$-modules, where $D = \beta(S \otimes_R \overline{R})$ and $\sigma : \overline{R} \to \overline{R}\otimes_R\overline{R}$ is an $R$-linear map defined by $\sigma(\alpha)=\alpha \otimes 1 - 1 \otimes \alpha$. Therefore the equality $\beta(S \otimes_RR^*) = S^*$ holds in $\rmQ(S)$. As a consequence, $S$ is strictly closed if and only if $S \otimes_R R^* = S \otimes_R R$. The latter condition is equivalent to saying that $R$ is strictly closed. This completes the proof. 
\end{proof}

Closing this section, we explore one example. 

\begin{ex}
Let $F[t]$ be the polynomial ring over a field $F$ of characteristic $3$. We set $K = \rmQ(F[t])$ and $k = \rmQ(F[t^3])$. Then $K = k[t]$ and $\{1, t, t^2\}$ forms a $k$-basis of $K$. Let $V = K[X]$ be the polynomial ring over $K$ and consider 
$$
R = k[X, tX, t^2X]
$$
in $V$. Then $R$ is strictly closed, but $K \otimes_k R$ is not strictly closed.  
\end{ex}

 \begin{proof}
 Notice that $\rmQ(V) = \rmQ(R)$, $V=k[t, X]=\overline{R}$, and $V$ is a module-finite extension of $R$. We set $N=XV$ and $M=XV \cap R$. Then $N=M \in \Max R$.
 Hence, by Proposition \ref{2.2}, we get 
$$
R^* \subseteq R + MV = R + XV = R + M \subseteq R
$$
which yields that the equality $R^* = R$ holds in $V$. Hence $R$ is strictly closed. We now consider the extension
$$
S = K \otimes_k R \subseteq T=K \otimes_kV \subseteq K \otimes_k\rmQ(V) = K \otimes_k\rmQ(R)
$$
of rings. Since $S = K \otimes_k R$ is Artinian and the inclusion $S = K \otimes_k R \to K \otimes_k \rmQ(R)$ is flat, we obtain $\rmQ(S) = K \otimes_k\rmQ(V)$. 

Let $z = t\otimes 1 - 1 \otimes t \in T=K \otimes_kV$. We then have $z^3=0$ and $z \in \overline{S}$. Hence
$$
\overline{S} \supseteq S + z \cdot \rmQ(S).
$$
Moreover we have $1\otimes t \not\in S$, because $t \not\in R$. Let $P =MS$. Then $P \in \Max S$ and $P = (1 \otimes X, 1 \otimes tX, 1 \otimes t^2X)S$. Since $M^2=XM$, we get $P^2 = (1 \otimes X)P$. Therefore
$$
P :P = \frac{P}{1 \otimes X} = S[1\otimes X] = K\otimes_kV=T.
$$
By \cite[Corollary 9.2]{CCCEGIM}, if $S$ is strictly closed, then $T=P:P$ is a weakly Arf ring. To show that $S$ is not strictly closed, it is enough to check that $T$ is not weakly Arf. 

Notice that $T = (K \otimes_k K)[1 \otimes X]$ and $1 \otimes X$ is transcendental over $K \otimes_k K$. We denote by  $B=A[Z]$ the polynomial ring over $A = K \otimes_k K$. Let $k[Y]$ be the polynomial ring over $k$ and consider the $k$-algebra map
$
\varphi : k[Y] \to K
$
such that $\varphi(Y) = t$. Then, since $\Ker \varphi = (Y^3-t^3)$, we have an isomorphism
$$
k[Y]/(Y^3-t^3) \cong K
$$
of $k$-algebras. 
We now regard $A$ as a $K$-algebra via the map $K \to A=K\otimes_kK, \alpha \mapsto \alpha \otimes 1$. Therefore, we get the isomorphisms
\begin{eqnarray*}
A = K \otimes_k K &\cong& K \otimes_k\left[k[Y]/(Y^3-t^3)\right] \\ 
&\cong& (K \otimes_kk[Y])/((1\otimes Y - t\otimes 1)^3) \\
&\cong& (K[Y])/((Y - t)^3) 
\end{eqnarray*}
of $K$-algebras, where $K[Y]$ denotes the polynomial ring over $K$. Consequently 
$$
B \cong K[Z, Y]/ ((Y-t)^3)
$$
as a ring. This implies $B$ is not a weakly Arf ring. Hence $T$ is not weakly Arf. 
 \end{proof}


\section{Relation between weakly Arf closures}

In this section, we investigate a relation between strict and weakly Arf closures. 
Theorem \ref{1.2} claims that the equality $R^a = R^*$ holds if $R$ is a Noetherian ring such that $\height_R M\geq 2$ for every $M \in \Max R$, $\overline{R}$ is a module-finite extension of $R$, and $R^a$ satisfies the condition $(S_2)$ of Serre. This section aims at exploring two concrete examples in order to examine a possibility for generalization of Theorem \ref{1.2}. 

We begin with the construction of weakly Arf closure. Let $W(R)$ be the set of non-zerodivisors on a ring $R$. 

\begin{defn}[{\cite[Definition 7.4]{CCCEGIM}}]
For an arbitrary commutative ring $R$, we set 
$$
R_1 = R\left[\frac{yz}{x} ~\middle|~ x \in W(R), y, z \in A \ \text{such that} \ \frac{y}{x}, \frac{z}{x} \in \overline{R} \ \right] \subseteq \rmQ(R)
$$
which is an intermediate ring between $R$ and $\overline{R}$. For each $n \ge 0$, we define recursively
$$
R_n=
\begin{cases}
\ R & (n=0),\\
\ \left[R_{n-1}\right]_1 & (n >0).
\end{cases}
$$
Notice that, for each $n \ge 0$, $\rmQ(R_n) = \rmQ(R)$, $\overline{R_n} = \overline{R}$, and we have a tower
$$
R=R_0 \subseteq R_1 \subseteq \cdots \subseteq R_n \subseteq \cdots \subseteq \overline{R}
$$
of rings. We define 
$$
R^a = \bigcup_{n \ge 0} R_n
$$ and call it {\it the weakly Arf closure} of $R$.
\end{defn}

It is proved in \cite[Proposition 7.5]{CCCEGIM} that $R^a$ is a weakly Arf ring, and $R^a \subseteq S$ for every intermediate ring $R \subseteq S \subseteq \overline{R}$ such that $S$ is  weakly Arf. In particular, since $R^*$ is a weakly Arf ring, we obtain $R^a \subseteq R^*$ where $R^*$ denotes the strict closure of $R$ in $\overline{R}$ (\cite[Proposition 4.5]{L}, \cite[Theorem 4.5]{CCCEGIM}). 

The first example is stated as follows.  

\begin{thm}\label{4.5}
Let $U = k[X, Y]$ be the polynomial ring over a field $k$, and we set $R = k[X^n, XY^{n-1}, Y^n]\subseteq U$ $(n\geq 4)$. Then the following assertions hold true.
\begin{enumerate}[$(1)$]
\item
$\overline{R} = k[X^{n-i}Y^i \mid 1\leq i \leq n ]$ and $\overline{R}  = R + \sum_{i=1}^{n-2} RX^{n-i}Y^i$. 
\item The strict closure $R^*$ of $R$ in $\overline{R}$ is given by
$$
R^* =  
\begin{cases}
\ R + RX^7Y^5 & (n = 4),\\
\ R + \sum_{i=1}^{n-3} RX^{2n-i}Y^{n+i} + \sum_{i=1}^{n-4} RX^{n-i}Y^{2n+i} & (n\geq 5).
\end{cases}
$$
\item
$R$ is a Cohen-Macaulay ring, but $R^*$ is not Cohen-Macaulay. 
\end{enumerate}
\end{thm}

\begin{proof}
We consider $U$ as a $\bbZ^2$-graded ring under the grading $U_{\mathbf 0}=k$, $X\in U_{\mathbf e_1}$, and $Y\in U_{\mathbf e_2}$ 
where ${\mathbf e_1}=
\left(\begin{smallmatrix}
1 \\
0
\end{smallmatrix}\right)$
and 
${\mathbf e_2}=
\left(\begin{smallmatrix}
0 \\
1
\end{smallmatrix}\right)$.
By setting 
$$
H=\{\left(\begin{smallmatrix}
a \\
b
\end{smallmatrix}\right)\in \bbZ^2
\mid a, b\ge 0, \ X^aY^b\in R \},
$$
we see that $R=k[H]$ is a graded subring of $U$. We set $S = k[X^{n-i}Y^i \mid 1\leq i \leq n ]$.

 $(1)$  Since $\frac{Y}{X}=\frac{Y^n}{XY^{n-1}}\in\rmQ(R)$, we have $R\subseteq S \subseteq \rmQ(R)$. This implies $\rmQ(S)=\rmQ(R)$. Moreover, because $S=R + \sum_{i=1}^{n-2} RX^{n-i}Y^i$ and $S$ is a normal domain, we get $S=\overline{R}$.

$(2)$  We now consider a presentation 
$$
R^{\oplus q} \overset{\Bbb M}{\longrightarrow} R^{\oplus (n-1)} \overset{\varepsilon}{\longrightarrow} S \longrightarrow 0
$$
of $S$ as an $R$-module, where $q > 0$ and $\varepsilon = \left[1 \ X^{n-1}Y^1\ X^{n-2}Y^2 \ \cdots \ X^2Y^{n-2}\right]$.
Then, by Theorem \ref{2.6}, we have 
$$
R^*\subseteq R+JS \ \ \text{in} \ \ S
$$ 
where $J$ denotes the ideal of $S$ generated by all the entries of the first row of $\Bbb M$. 

\begin{claim}\label{1}
$JS=(X^nY^n, (XY^{n-1})^2)S$.
\end{claim}
\begin{proof}[Proof of Claim \ref{1}]
Notice that $J$ is a graded ideal of $R$, because $J=R\cap \left( \sum_{i=1}^{n-2} RX^{n-i}Y^i \right)$. Since
$$\varepsilon\left(\begin{smallmatrix}
X^nY^n \\
-XY^{n-1} \\
0 \\
\vdots \\
0
\end{smallmatrix}\right) = 0\ \text{and}\ 
\varepsilon\left(\begin{smallmatrix}
(XY^{n-1})^2 \\
0 \\
\vdots \\
0 \\
-Y^n
\end{smallmatrix}\right) = 0,
$$
we then have $(X^nY^n, (XY^{n-1})^2)S\subseteq JS$. Conversely, for each $a\in J$, we show that $a\in (X^nY^n, (XY^{n-1})^2)S$. Indeed, we may assume $a\neq 0$ and $a$ is a homogeneous element. Let us write $a=X^\alpha Y^\beta$ with  $\left(\begin{smallmatrix}
\alpha \\
\beta
\end{smallmatrix}\right)\in H$. Then, because $a=X^\alpha Y^\beta \in  \sum_{i=1}^{n-2} RX^{n-i}Y^i$, 
$$
\left(\begin{smallmatrix}
\alpha \\
\beta
\end{smallmatrix}\right)=\left(\begin{smallmatrix}
\alpha_1 \\
\beta_1
\end{smallmatrix}\right)+
\left(\begin{smallmatrix}
n-i \\
i
\end{smallmatrix}\right)=
\left(\begin{smallmatrix}
\alpha_1+n-i \\
\beta_1+i
\end{smallmatrix}\right)$$
for some $\left(\begin{smallmatrix}
\alpha_1 \\
\beta_1
\end{smallmatrix}\right)\in H$
and $1\le i\le n-2$.
In addition, because $a=X^\alpha Y^\beta \in R$, we  write 
$$\left(\begin{smallmatrix}
\alpha \\
\beta
\end{smallmatrix}\right)=
\gamma_1\left(\begin{smallmatrix}
n \\
0
\end{smallmatrix}\right)+
\gamma_2\left(\begin{smallmatrix}
1 \\
n-1
\end{smallmatrix}\right)+
\gamma_3\left(\begin{smallmatrix}
0 \\
n
\end{smallmatrix}\right)
=
\left(\begin{smallmatrix}
\gamma_1n+\gamma_2 \\
\gamma_2(n-1)+\gamma_3n
\end{smallmatrix}\right)
$$
for some $\gamma_i\ge 0$. We now assume $a\notin (X^nY^n, (XY^{n-1})^2)S$, and seek a contradiction. Notice that $\gamma_2\le 1$.
If $\gamma_3\neq 0$, then we have $\gamma_1=0$, because $a\notin X^nY^nS$. This implies 
$$
\gamma_2=\alpha=\alpha_1+n-i\ge 2
$$
which is impossible. Thus $\gamma_3=0$. In contrast, because $\beta=\gamma_2(n-1)=\beta_1+i\ge 1$, we have $\gamma_2=1$ and $\beta=n-1=\beta_1+i$. However, since $\left(\begin{smallmatrix}
\alpha_1 \\
\beta_1
\end{smallmatrix}\right)\in H$, $\beta_1=0$ or $\beta_1\ge n-1$, which contradicts $n-1=\beta_1+i$ and $1\le i\le n-2$. Therefore, $a\in (X^nY^n, (XY^{n-1})^2)S$, as desired.
\end{proof}

To sum up this argument, we obtain
\begin{eqnarray*}
R^* & \subseteq &R+(X^nY^n, (XY^{n-1})^2)S \\
& =&\begin{cases}
\ R + RX^7Y^5 & (n = 4)\\
\ R + \sum_{i=1}^{n-3} RX^{2n-i}Y^{n+i} + \sum_{i=1}^{n-4} RX^{n-i}Y^{2n+i} & (n\geq 5)
\end{cases}.
\end{eqnarray*}

On the other hand, by applying the functor $ S \otimes_R (-)$ to the presentation of $S$, we have the following diagram 
$$
\xymatrix{
 S \otimes_R R^{\oplus(n-1)} \ar[r]^{\ \ \ 1 \otimes \varepsilon}  & S \otimes_R S  \ar[r]^{} &  0\\
S^{\oplus(n-1)} \ar[u]^{\cong} \ar[ru]^{\sigma} & & 
}
$$
of $S$-modules. We define  
$$
Z_i=X^{2n-i}Y^{n+i}{\mathbf e_0}-X^{n+1}Y^{n-1}{\mathbf e_{i+1}},\ \text{and}\ \ W_j=X^{n-j}Y^{2n+j}{\mathbf e_0} - Y^{2n}{\mathbf e_j}\ \ \ \ \text{in $S^{\oplus (n-1)}$}
$$
for $1\le i\le n-3$ and $1\le j\le n-4$. Then 
\begin{eqnarray*}
\sigma(Z_i) &=& X^{2n-i}Y^{n+i}\otimes 1-1\otimes X^{2n-i}Y^{n+i}\\
\sigma(W_j) &=& X^{n-j}Y^{2n+j}\otimes 1-1\otimes X^{n-j}Y^{2n+j}
\end{eqnarray*}
in $S\otimes_R S$. To show the converse  $R+JS \subseteq R^*$, it is enough to prove that $Z_i, W_j\in \Ker \sigma$ for every $1\le i \le n-3$ and $1\le j\le n-4$.
Indeed, notice that the elements of the form
$$
X^nY^n{\mathbf e_0}-XY^{n-1}{\mathbf e_1},\  \ X^2Y^{2(n-1)}{\mathbf e_0}-Y^n{\mathbf e_{n-2}}, \ \ Y^n{\mathbf e_k}- XY^{n-1}{\mathbf e_{k+1}}
$$
are in $\Ker \varepsilon$ for every $1\le k\le n-3$. Then, because
\begin{eqnarray*}
Z_i &=&X^{n-i}Y^i(X^nY^n{\mathbf e_0}-XY^{n-1}{\mathbf e_1})+\sum_{k=1}^{i}X^{n-i+k}Y^{i-k}(Y^n{\mathbf e_k}- XY^{n-1}{\mathbf e_{k+1}}) \\
W_j &=& X^{n-(j+2)}Y^{j+2}(X^2Y^{2(n-1)}{\mathbf e_0}-Y^n{\mathbf e_{n-2}})-\sum_{k=1}^{n-2-j}X^{k-1}Y^{n+1-k}(Y^n{\mathbf e_{k+j-1}}- XY^{n-1}{\mathbf e_{k+j}})
\end{eqnarray*}
for any $1\le i \le n-3$ and $1\le j\le n-4$, we conclude that $Z_i, W_j\in \Ker \sigma$.

$(3)$  Since $X^n, Y^n$ forms a regular sequence on $R$, $R$ is a Cohen-Macaulay ring of dimension $2$. We have $X^{2n-1}Y^{n-1}\in R^*$ and 
$Y^n\cdot X^{2n-1}Y^{n+1}=X^n\cdot X^{n-1}Y^{2n+1}\in X^nR^*$,
but $X^{2n-1}Y^{n+1}\notin X^nR^*$. Therefore, $R^*$ is not a Cohen-Macaulay ring.
\end{proof}

 Consequently, we get the following example in which $R^a = R^*$ holds, even if $R^a$ does not satisfy the condition $(S_2)$ of Serre.  

\begin{ex}\label{3.4}
Let $k[X, Y]$ be the polynomial ring over a field. We consider $R_1 = k[X^4, XY^3, Y^4]$ and $R_2 = k[X^5, XY^4, Y^5]$. Then $R_i \subsetneq [R_i]^a = [R_i]^*$ for each $1 \leq i \leq 2$. 
\end{ex}

\begin{proof}
By Theorem \ref{4.5}, we have
\begin{eqnarray*}
{[R_1]}^* &=& R_1 + R_1X^7Y^5, \\
{[R_2]}^* &=& R_2+R_2X^9Y^6+R_2 X^8Y^7+R_2 X^4Y^{11}.
\end{eqnarray*}
Since $X^7Y^5=\frac{(X^4Y^4)^2}{XY^3}\in [R_1]^a$, we get $[R_1]^a=[R_1]^*$. Moreover, because $X^9Y^6=\frac{(X^5Y^5)^2}{XY^4}$, $X^4Y^{11}=\frac{(XY^4)^4}{XY^3}$, and $X^{13}Y^7=\frac{(X^9Y^6)^2}{X^5Y^5}$, we have 
$$
X^9Y^6, X^4Y^{11}, X^{13}Y^7\in[R_2]^a.
$$ 
Let $\xi=X^4Y-X^3Y^2+X^2Y^3+X^8Y^2-X^7Y^3\in \overline{R}$. Then $(XY^4+Y^5)\xi=X^5Y^5+X^9Y^6+X^2Y^8-X^7Y^8\in [R_2]^a$, so that
\begin{eqnarray*}
\frac{((XY^4+Y^5)\xi)^2}{XY^4+Y^5}&=&(XY^4+Y^5)\xi^2\\
&=&X^{9}Y^{6}-X^{8}Y^{7}+X^{7}Y^{8}+2X^{13}Y^{7}-2X^{12}Y^{8}+X^{17}Y^{8}-X^{16}Y^{9}\\
&+&X^{6}Y^{9}-X^{5}Y^{10}+X^{4}Y^{11}+2X^{10}Y^{10}-X^{9}Y^{11}-X^{15}Y^{10}+X^{14}Y^{11}\\
&\in& [R_2]^a.
\end{eqnarray*}
This implies $X^8Y^7\in [R_2]^a$, and hence $[R_2]^a=[R_2]^*$.
\end{proof}

Closing this paper we explore the second example. This shows, even if the base ring $R$ satisfies the condition $(S_2)$ of Serre, the equality $R^a = R^*$ does not hold in general.

Let $\rmH_1(R)$ be the set of all height one prime ideals of a Noetherian ring $R$. For a Noetherian local ring $R$, we denote by $v(R)$ (resp. $\rme(R)$) the embedding dimension (resp. the multiplicity) of $R$.

\begin{thm}
Let $U = k[X, Y, Z]$ be the polynomial ring over a field $k = \Bbb Z/(2)$ and consider $R = U/(XY(X+Y))$. We denote by $x, y, z$ the images of $X, Y, Z$ in $R$, respectively. We set $A = k[x, y]$. Then the following assertions hold true.
\begin{enumerate}[$(1)$]
\item
$\overline{R} = R/xR\times R/yR \times R/(x+y)R$ and $\overline{A} = A/xA \times A/yA \times A/(x+y)A$.
\item The strict closure $A^*$ of $A$ in $\overline{A}$ is given by 
$$
A^* = A + \m\overline{A} = A + k\rho
$$ and $\ell_A(A^*/A) = 1$, where $\m = (x, y)A$ and $\rho = (\overline{y}, 0, 0) \in \overline{A}$.
\item
$A = A^a \subsetneq A^*$.
\item
$R^* = A^*[z] = R + k[z]\cdot \rho$ and $R^a = R + k[z]\cdot z(1+z)\rho$. Hence, $R\subsetneq R^a \subsetneq R^*$.
\end{enumerate}
\end{thm}

\begin{proof}
The assertion $(1)$ follows from the fact that both rings $R$ and $A$ are reduced. Passing to the canonical injections  
$$A/xA\to R/xR,\ \ A/yA\to R/yR,\ \ \text{and}\ \ A/(x+y)A\to R/(x+y)R,$$
we may assume that $\overline{A} \subseteq \overline{R}$.

$(2)$ Note that $\m \in \rmH_1(A)$. Since $v(A_\m)=2<3={\rme}(A_\m)$, the local ring $A_\m$ is not Arf. Therefore, because $A$ is a Cohen-Macaulay ring of dimension one, $A\neq A^*$ by \cite[Theorem 4.5]{CCCEGIM}. Moreover, by Corollary \ref{2.3}, we have 
$$
A^*\subseteq A+\m\overline{A}=A+k\rho
$$
where $\rho=(\overline{y}, 0, 0) \in \overline{A}$. Since $\ell_A(A+k\rho/A)=1$ and $A\subsetneq A^*\subseteq A+k\rho$, we get $A^*=A+\m\overline{A}=A+k\rho$ and $\ell_A(A^*/A)=1$.

$(3)$ Notice that $A:\overline{A}=\m^2$. In fact, since $\overline{A}=A/xA\times A/yA \times A/(x+y)A$,  for each $\alpha\in A$, we have $\alpha \overline{A}\subseteq A$ if and only if there exists $a, b, c \in A$ such that 
\begin{eqnarray*}
&\alpha-a\in xA,\ \ a\in yA\cap (x+y)A=y(x+y)A,\\
&\alpha-b\in yA,\ \ b\in xA\cap (x+y)A=x(x+y)A,\\
&\alpha-c\in (x+y)A,\ \ \text{and} \ \  c\in xA\cap yA=xyA.
\end{eqnarray*}  
The latter condition is equivalent to saying that $\alpha\in (x, y^2)A\cap (y, x^2)A \cap (xy, x+y)A=\m^2$. Hence $A:\overline{A}=\m^2$ as claimed.  
Therefore, for each $\fkp \in {\rmH}_1(A)$ with $\fkp\neq \m$, $A_\fkp=\overline{A_\fkp}$, and hence $A_\fkp$ is a weakly Arf ring. 
Furthermore, since the $\m$-adic completion $\widehat{A_{\m}}=k[[x, y]]$ of $A_\m$ is a weakly Arf ring, $A_\m$ is also weakly Arf (\cite[Proposition 3.3, Example 9.6]{CCCEGIM}). Hence, by \cite[Theorem 2.6]{CCCEGIM}, $A$ is a weakly Arf ring.

$(4)$  
Since $z\in R$ is transcendental over $A$ and $\overline{R}\subseteq \rmQ(A)[z]$, we have $\overline{R}=A[z]$ and
$$
R^*=A^*[z]=R+k[z]\cdot \rho
$$
by \cite[Lemma 4.9]{CCCEGIM}. Let $\alpha=x+yz \in R$ and $\xi=(1, \overline{z}, 0) \in \overline{R}$. Then $\alpha\in W(R)$ and $\alpha\xi=(\overline{yz}, \overline{xz}, 0)=(x+y)z\in R$, where $\overline{*}$ is the image of $*$ in the corresponding ring. Therefore, we have
$$
\frac{(\alpha\xi)^2}{\alpha}=\alpha\xi^2=(\overline{yz}, \overline{xz^2}, 0)=(x+y)z^2+z(1+z)\rho \in R^a
$$
so that $R+k[z]\cdot z(1+z)\rho\subseteq R^a$.

To show the converse,  we need the following. 


\begin{lem}\label{4.8}
Let $R$ be an arbitrary commutative ring, and let $\{R_\lambda\}_{\lambda\in \Lambda}$ be a family of intermediate rings between $R$ and  $\overline{R}$ the integral closure of $R$ in $\rmQ(R)$. If $R_\lambda$ is a weakly Arf ring for every $\lambda\in \Lambda$, then so is $\underset{\lambda\in\Lambda}{\bigcap} R_\lambda$.
\end{lem}

\begin{proof}
Let $B=\underset{\lambda\in\Lambda}{\bigcap} R_\lambda$. Then $\rmQ(B)=\rmQ(R)$ and $\overline{B}=\overline{R}$. Let $a, b, c \in B$ such that $a\in W(B)$ and $\frac{b}{a}, \frac{c}{a}\in \overline{B}=\overline{R}$. 
We then have $a\in W(R_\lambda)$, $b, c \in R_\lambda$, and $\frac{b}{a}, \frac{c}{a}\in \overline{R}=\overline{R_\lambda}$ for every $\lambda\in \Lambda$. Because $R_\lambda$ is weakly Arf, we have $\frac{bc}{a}\in R_\lambda$ for every $\lambda\in\Lambda$. Therefore $\frac{bc}{a}\in B$, and hence $B$ is a weakly Arf ring. 
\end{proof}

Let $T_1=R+k[z]\cdot z\rho$, $T_2=R+k[z]\cdot (1+z)\rho$, and $z_1=1+z$. Then, because $T_2=R+k[z_1]\cdot z_1\rho$ and $z_1$ is transcendental over $k$, we have an isomorphism
$$
T_1\cong T_2
$$ of rings. Hence, if $T_1$ is a weakly Arf ring, then so is $T_2$. Lemma \ref{4.8} shows  $T_1\cap T_2=R+k[z]\cdot z(1+z)\rho$ is also a weakly Arf ring. Consequently, we get 
$$
T_1\cap T_2=R+k[z]\cdot z(1+z)\rho=R^a
$$
because $T_1\cap T_2\subseteq R^a$ and $T_1\cap T_2$ is a weakly Arf ring. 

Hence, it remains to show that $T_1$ is a weakly Arf ring. By setting  $M_1=(x, y, z, z\rho)T_1 \in \Max T_1$, we have an isomorphism
$$
R^*/T_1\cong T_1/M_1
$$ 
as a $T_1$-module. This implies $[T_1]_\fkp=[R^*]_\fkp$ for every $\fkp\in \Spec T_1$ with $\fkp\neq M_1$. Thus, $[T_1]_\fkp$ is a weakly Arf ring (remember that $R^*$ is a weakly Arf ring that satisfies the Serre condition $(S_1)$). 

Therefore, it is enough to show that $[T_1]_{M_1}$ is a weakly Arf ring. To do this, we need to check that 
\begin{center}
$(\overline{aT_1})^2=a(\overline{aT_1})$ for every $a \in M_1\cap W(T_1)$. 
\end{center}
Let $a \in M_1\cap W(T_1)$. Notice that $(\overline{aT_1})^2=a(\overline{aT_1})$ if and only if $\frac{\overline{aT_1}}{a}=\{ \varphi\in \overline{T_1} \mid a\varphi \in T_1 \}$ has a ring structure.
Since $a \rho\in T_1$, we then have 
$$
R^*\subseteq \frac{\overline{aT_1}}{a}.
$$
On the other hand, because $M_1\subseteq (x, y, \rho)R+zk[z]$, we can write 
\begin{center}
$a=b+gz$ with $b\in (x, y, \rho)R$ and $g\in k[z]$.
\end{center} 
If $\varphi \in \frac{\overline{aR^*}}{a}=\{ \varphi\in \overline{R^*} \mid a\varphi \in R^* \}$, then $a\varphi=b\varphi+ gz\varphi\in R^*$. As $b\varphi\in R^*$, we have $gz\varphi \in R^*$. 
If $g\neq 0$, then $\varphi \in R^*$, so that 
$$
\frac{\overline{aT_1}}{a}=\frac{\overline{aR^*}}{a}\subseteq R^*
$$
which implies $\frac{\overline{aT_1}}{a}=R^*$. Hence we may assume $g=0$. 

Hence, we can write $a=b=c_1x+c_2y+a_1$ with $c_1, c_2\in k$ and $a_1\in (xz, yz, z\rho)k[z]$. Then, for $\varphi \in\overline{T_1}$, $a\varphi\in T_1$ if and only if $(c_1x+c_2y)\varphi \in T_1$. 

Here, we need the following.

\begin{lem}
Let $R[t]$ be the polynomial ring over $R$. Then there is an isomorphism
$$
R[t]/(xt, yt-y(x+y)z, t^2-yzt)\cong T_1
$$
as an $R$-algebra.
\end{lem}
\begin{proof}
We consider the $R$-algebra map 
$$
\varphi: R[t]\to T_1=R+k[z]\cdot z\rho
$$
defined by $ \varphi(t) = z\rho$. We set $I=(xt, yt-y(x+y)z, t^2-yzt)$. Then $I\subseteq \Ker \varphi$. 
Let $\xi\in \Ker \varphi$. We write 
\begin{center}
$\xi=\overline{\alpha+\beta t}$ with $\alpha, \beta\in R$, and \\
$\beta=x\eta_1+y\eta_2+f$ with $\eta_1, \eta_2\in R$, $f\in k[z]$.
\end{center}
Because $\alpha+\beta\cdot z\rho=0$ in $T_1$, we have 
\begin{center}
$\overline{\beta\cdot z\rho}=\overline{f\cdot z\rho}=0$ in $T_1/R$
\end{center}
which implies $f=0$. Therefore, $\alpha=-\beta\cdot z \rho =-y(x+y)z\eta_2$ in $R$. Hence
$$
\alpha+\beta t=-y(x+y)z\eta_2+(x\eta_1+y\eta_2)t=\eta_1xt+\eta_2(yt-y(x+y)z)\in I.
$$
Consequently, $\xi=0$ in $R/I$ and $I=\Ker \varphi$, as desired.
\end{proof}

By using this isomorphism, we obtain
$$
\overline{T_1}=T_1/(x, yz+z\rho)T_1 \times T_1/(y, z\rho)T_1 \times T_1/(x+y, z\rho)T_1.
$$ 

Finally we reach the following, which guarantees that $T_1$ is weakly Arf. 

\begin{claim}\label{2}
We have
\begin{center}
$\frac{\overline{aT_1}}{a} =  
\begin{cases}
\ \overline{T_1} & (c_1=c_2=0)\\
\ R^*+{}_{T_1} \langle (1, 0, 0), (0, \overline{z}, 0) \rangle & (c_1=1,\ c_2=0 )\\
\ R^*+{}_{T_1} \langle (0, 0, 1), (\overline{z}, 0, 0) \rangle & (c_1=c_2=1 )\\
\ R^*+{}_{T_1} \langle (0, 1, 0), (0, \overline{z}, 0) \rangle & (c_1=0,\ c_2=1 )\\
\end{cases}
.$
\end{center}
Therefore, $(\overline{aT_1})^2=a(\overline{aT_1})$.
\end{claim}

\begin{proof}[Proof of Claim \ref{2}]
If $c_1=c_2=0$, then  $\frac{\overline{aT_1}}{a}=\overline{T_1}$. Suppose $c_1=1$ and $c_2=0$. Because 
$$
x\rho=0,\ \  x\cdot(1, 0, 0)=0,\ \ \text{ and}
$$
$$
x\cdot (0, \overline{z}, 0)=(0, \overline{xz}, 0)=(\overline{(x+y)z+z\rho}, \overline{(x+y)z+z\rho}, \overline{(x+y)z+z\rho})=(x+y)z+z\rho\ \ \text{ in}\ \  \overline{T_1},
$$ 
we have $R^*+{}_{T_1} \langle (1, 0, 0), (0, \overline{z}, 0) \rangle\subseteq \frac{\overline{aT_1}}{a} $. Conversely, let $\varphi \in \overline{T_1}$ such that $x\varphi \in T_1$.
Then, because $\overline{T_1}=R^*+k[z]\cdot (1, 0, 0)+k[z]\cdot (0, 1, 0)$, we write $\varphi=\psi+ (\overline{f_1}, \overline{f_2}, 0)$ with $\psi \in R^*$ and $f_1, f_2\in k[z]$. In addition, let us write $f_2=c+zf_3$ with $c\in k$ and $f_3\in k[z]$. Then, since $x\alpha, (0, \overline{xzf_3}, 0)\in T_1$, we have
$$
x\varphi=x\alpha+c\cdot (0, \overline{x}, 0)+(0, \overline{xzf_3}, 0) \in T_1
$$
so that $c\cdot (0, \overline{x}, 0)\in T_1$. Since $(0, \overline{x}, 0)\notin T_1$, we get $c=0$. Therefore, $\varphi=\psi+ (\overline{f_1}, \overline{zf_3}, 0)\in R^*+{}_{T_1} \langle (1, 0, 0), (0, \overline{z}, 0) \rangle$, which implies $\frac{\overline{aT_1}}{a}=R^*+{}_{T_1} \langle (1, 0, 0), (0, \overline{z}, 0) \rangle $. Similarly we can show in other cases as well. 
\end{proof}
This completes the proof of $(4)$.
\end{proof}



\end{document}